\documentclass{amsart}

\usepackage[ascii]{inputenc}
\usepackage{amsmath,amsfonts,amssymb,amsthm}
\usepackage{hyperref,graphicx,cleveref,enumitem}

\usepackage[dvipsnames]{xcolor}

\theoremstyle{plain}
\newtheorem{theorem}{Theorem}
\newtheorem{lemma}[theorem]{Lemma}
\newtheorem{proposition}[theorem]{Proposition}
\newtheorem{corollary}{Corollary}

\theoremstyle{remark}
\newtheorem{remark}{Remark}
\newtheorem*{remark*}{Remark}

\theoremstyle{definition}
\newtheorem*{definition*}{Definition}

\newtheorem*{problem*}{Problem}

\def\Lanoise{\Lambda_{\bar q}^\eta}
\def\C{\mathbb{C}}
\def\N{\mathbb{N}}
\def\R{\mathbb{R}}
\def\Z{\mathbb{Z}}

\def\H{{L^2(\T^d)}}

\def\W{\mathcal{W}}

\def\T{\mathbb{T}}
\def\supp{{\rm supp}}

\def\pv{{\rm p.v.}}

\renewcommand{\epsilon}{\varepsilon}
\def\q{q^\varepsilon_\eta}

\newcommand{\norm}[1]{{\left\Vert {#1}\right\Vert}}

\begin{document}
\title[Calder\'on's inverse problem with finitely many measurements II]{Calder\'on's Inverse Problem\\ with a Finite Number of Measurements II: independent data}

\thanks{This work has been carried out at the Machine Learning Genoa (MaLGa) center, Universit\`a di Genova (IT). The authors are members of the ``Gruppo Nazionale per l'Analisi Matematica, la Probabilit\`a e le loro Applicazioni'' (GNAMPA), of the ``Istituto Nazionale per l'Alta Matematica'' (INdAM). GSA is supported by a UniGe starting grant ``curiosity driven''.}

\author{Giovanni S. Alberti}
\address{MaLGa Center, Department of Mathematics, University of Genoa, Via Dodecaneso 35, 16146 Genova, Italy.}
\email{giovanni.alberti@unige.it}

\author{Matteo Santacesaria}
\address{MaLGa Center, Department of Mathematics, University of Genoa, Via Dodecaneso 35, 16146 Genova, Italy.}
\email{matteo.santacesaria@unige.it}

%,   42C40, 94A20}

\date{March 9, 2020}

\dedicatory{Dedicated to Sergio Vessella on the occasion of his 65\textsuperscript{th} birthday}

\begin{abstract}
We prove a local Lipschitz stability estimate for Gel'fand-Cal\-der\'on's inverse problem for the Schr\"odinger equation. The main novelty is that only a finite number of boundary input data is available, and those are independent of the unknown potential, provided it belongs to a known finite-dimensional subspace of $L^\infty$. A similar result for Calder\'on's problem is obtained as a corollary. This improves upon two previous results of the authors on several aspects, namely the number of measurements and the stability with respect to mismodeling errors. A new iterative reconstruction scheme based on the stability result is also presented, for which we prove exponential convergence in the number of iterations and stability with respect to noise in the data and to mismodeling errors.
\end{abstract}


\maketitle

\section{Introduction}

Consider the Schr\"odinger equation
\begin{equation}\label{eq:schr}
(-\Delta +q ) u = 0 \qquad \text{in } \Omega,
\end{equation}
where $\Omega \subseteq \R^d$, $d \ge 3$, is an open bounded domain with Lipschitz boundary and $q \in L^{\infty}(\Omega)$ is a potential. Assuming that 
\begin{equation}\label{hyp:dir}
0\; \text{is not a Dirichlet eigenvalue for } -\Delta + q  \text{ in } \Omega,
\end{equation}
 it is possible to define the Dirichlet-to-Neumann (DN) map
\begin{equation*}
\Lambda_q \colon H^{1/2}(\partial \Omega) \to H^{-1/2}(\partial \Omega),\qquad  u|_{\partial \Omega} \mapsto \left.\frac{\partial u}{\partial \nu}\right|_{\partial \Omega},
\end{equation*}
where $\nu$ is the unit outward normal to $\partial \Omega$.
Gel'fand-Calder\'on's inverse problem consists of the reconstruction of $q$ from the knowledge of the associated DN map $\Lambda_q$. 

Thanks to a change of variables (see, e.g., \cite{Sylvester1987}), this can be seen as a generalization of Calder\'on's inverse conductivity problem \cite{calderon1980}, where one wants to determine a conductivity distribution $\sigma \in L^\infty(\Omega)$ satisfying
\begin{equation}\label{eq:lambda}
\lambda^{-1} \leq \sigma \leq \lambda \quad \text{almost everywhere in } \Omega
\end{equation}
for some $\lambda>1$, from the DN map
\[
\Lambda_\sigma\colon u|_{\partial \Omega} \mapsto \sigma \left.\partial_\nu u\right|_{\partial \Omega},
\]
where $u$ solves the conductivity equation $-\nabla \cdot (\sigma \nabla u) = 0$ in $\Omega$. For further details, the reader is referred to the review papers \cite{eit-1999,2002-borcea,2009-uhlmann,adler-gaburro-lionheart-2015,uhlmann-2019} and to the references therein.

The DN maps represent an infinite number of boundary measurements, a clearly unrealistic setting. In our previous work \cite{alberti2018} we  showed uniqueness and Lipschitz stability for both inverse problems when only a finite number of measurements is available, under the assumption that $q$ (or $\frac{\Delta\sqrt{\sigma}}{\sqrt{\sigma}}$) belongs to a known finite-dimensional subspace $\W$ of $L^\infty(\Omega)$. Note that Lipschitz stability results were previously known only with infinitely many measurements \cite{2005-alessandrini-vessella,2011-beretta-francini,beretta2013,beretta2015,gaburro2015,beretta2017,alessandrini2017,
alessandrini2017u,alessandrini2018,beretta2019}. In the case of a finite number of measurements, few Lipschitz stability results have recently appeared \cite{Harrach_2019,ruland2018,alberti2019infinite,harrach2019uniqueness}.

The main drawback of \cite{alberti2018} is that the boundary input data depend on the unknown to be reconstructed (even though their number is given a priori). This issue was solved in \cite{Harrach_2019}  for electrical impedance tomography (EIT), at the price of a non constructive choice for the number of measurements. In our recent work \cite{alberti2019infinite}, we showed that it is possible to give a priori both the number and the type of measurements for a large class of inverse problems. However, with respect to \cite{alberti2018}, this approach yields a higher number of measurements and a worse Lipschitz stability constant in case of mismodeling errors, namely, when $q$ (or $\frac{\Delta\sqrt{\sigma}}{\sqrt{\sigma}}$) is not exactly in $\W$.

In the present work we continue along the approach of \cite{alberti2018}, based on the nonlinear  method tailored for Calder\'on's problem. We consider both Gel'fand-Calder\'on's and Calder\'on's problems, and prove a local Lipschitz stability result and derive a nonlinear iterative reconstruction algorithm, with few measurements given a priori and a good stability with respect to mismodeling errors (and so, keeping the best aspects of the previous results). We also prove that the reconstruction algorithm is stable with respect to noise in the data, which can be seen as a first step towards a new regularization strategy for these problems. This is achieved at the expense of a local argument, namely, it is assumed that the unknown $q$ is sufficiently close to a known potential $q_0$.

The paper is structured as follows. In Section~\ref{sec:main} we state the main results regarding uniqueness and Lipschitz stability. Their proofs are given in Section~\ref{sec:uniqueness}. Section~\ref{sec:recon} is devoted to a new nonlinear reconstruction algorithm for which we prove exponential convergence and stability with respect to noise and to mismodeling errors. The main technical lemmata needed to prove the main results are in Section~\ref{sec:layer} and concern some properties of generalized single and double layer operators.

\section{Main results}\label{sec:main}

It is useful to recall the main result of our previous work \cite{alberti2018}. We first focus on Gel'fand-Calder\'on's problem.
Without loss of generality, we assume that $\Omega\subseteq \mathbb{T}^d$, $d \geq 3$, where $\mathbb{T}=[0,1]$. 
In the following we will extend any function of $L^\infty(\Omega)$ to $L^\infty(\R^d)$ by zero. We assume the a priori upper estimate
$
\norm{q}_\infty \le R
$
for some $R>0$ and that $q$ is well-approximated by  $\W$,  a fixed finite-dimensional subspace of $L^\infty(\Omega)$. 

The method is based on a particular class of solutions to \eqref{eq:schr}, called complex geometrical optics (CGO) solutions \cite{Sylvester1987} (see also \cite{Faddeev1965}).

For  $k \in \Z^d$, choose $\eta, \xi \in \R^d$ such that $|\xi| = |\eta|= 1$ and $\xi \cdot \eta = \xi \cdot k = \eta \cdot k = 0$. For $t\in\R$ define
\begin{equation}\label{eq:zeta}
\zeta^{k,t} = -i (\pi  k + t\xi) + \sqrt{t^2 +\pi^2|k|^2}\eta.
%\qquad \zeta_2^{k,t} = -i(\pi k - t\xi) - \sqrt{t^2 +\pi^2|k|^2}\eta,
\end{equation}
For every $t\ge c_1$, where $c_1=c_1(R)$ is given in Lemma~\ref{lem:Sqbound} below, we can construct a solution $\psi^{k,t}$ of \eqref{eq:schr} in $\R^d$ (with $q$ extended to $\R^d$ by zero) of the form
\[
\psi_q^{k,t}(x)  = e^{\zeta^{k,t}\cdot x}(1+r_q^{k,t}(x)),\qquad x\in \R^d,
\]
where the remainder term $r_q^{k,t}$ satisfies suitable decay estimates \cite{alberti2018}.
We consider an ordering of the frequencies in $\Z^d$, namely a bijective map $\rho\colon\N\to \Z^d$, $n\mapsto k_n$, such that
\begin{equation}\label{cond:rho}
|k_n| \leq  C_\rho \,n^{1/d},\qquad n\in\mathbb{N},
\end{equation} 
for some  $C_\rho > 0$. Here and in the following we use the notation $\N=\{1,2,\dots\}$. Set $t_n= c(|k_n|^{d}+1)$, where $c\ge c_1$ is a sufficiently large positive constant depending only on $R$. We use the notation
\begin{equation}\label{eq:notation}
\zeta_n=\zeta^{k_n,t_n},\qquad  \psi^n_q = \psi_q^{k_n,t_n},\qquad f^n_q = \psi^n_q|_{\partial\Omega}.
\end{equation}

%While the precise construction is postponed to Section~\ref{sec:uniqueness}, we introduce here the main notation needed to state the results. The CGO solutions depend on $q$ and on a parameter $\zeta\in\C^d$. After choosing a suitable countable set of parameters $\{\zeta_n\}_{n\in\N}$ (done a priori), it is possible to construct the corresponding CGO solutions, which will be denoted by $\psi^n_q$. Their traces on the boundary will be denoted by
%\begin{equation}\notag
 %f^n_q = \psi^n_q|_{\partial\Omega}.
%\end{equation}

The main stability result of \cite{alberti2018}, giving Lipschitz stability (and uniqueness) with a finite number of measurements, reads as follows. We use the following notation. Let $P_\W\colon L^2(\T^d)\to L^2(\T^d)$ be the orthogonal projection onto $i(\W)$, where $i\colon L^\infty(\Omega)\to L^2(\T^d)$ is the extension operator by zero. We also set $P_\W^\perp=I-P_\W$. 
 
\begin{theorem}[{\cite[Theorem 2 and Remark 6]{alberti2018}}]\label{thm:stab0}
Take $d\ge 3$ and let $\Omega \subseteq \mathbb{T}^d$ be a  bounded Lipschitz domain and $\W \subseteq L^\infty(\Omega)$ be a finite-dimensional subspace. There exists $N\in\N$ such that the following is true.

For every
$R,\epsilon>0$ and $q_1,q_2 \in L^\infty(\Omega)$ satisfying \eqref{hyp:dir} and
\begin{equation}\label{eq:q1q2}
\norm{q_j}_\infty\le R\quad\text{and}\quad \norm{P_\W^\perp q_j}_{L^2(\T^d)}\le \epsilon,\qquad j=1,2,
\end{equation}
 we have
\begin{equation*}
\|q_2 - q_1\|_{L^2(\Omega)} \leq e^{C N}
\norm{\left(\Lambda_{q_2}f^n_{q_1} - \Lambda_{q_1}f^n_{q_1}\right)_{n=1}^N}_{H^{-1/2}(\partial \Omega)^N} +8\epsilon
\end{equation*}
for some $C>0$ depending only on $\Omega$ and $R$. \end{theorem}
\begin{remark*}
It is worth observing that the stability constant $e^{C N}$ may be lowered to $e^{C N^{\frac{1}{2}+\alpha}}$ for any fixed parameter $\alpha>0$. However, in this paper we will not track the dependence of the stability constants on $N$ precisely, and so we opted for this simplified version.
\end{remark*}

The main drawback of this result is that, even if only finitely many boundary measurements are used (and the number of measurements $N$ is given a priori), these still depend on the unknown potential $q_1$. The main result of this work states that \textit{local} uniqueness and stability hold with boundary values given a priori.

\begin{theorem}\label{thm:stab}
Take $d\in\{3,4\}$, $\epsilon>0$ and let $\Omega \subseteq \mathbb{T}^d$ be a  bounded Lipschitz domain with connected complement, $\W \subseteq L^\infty(\Omega)$ be a finite-dimensional subspace and $N$ be as in Theorem~\ref{thm:stab0}.
Take $R>0$ and $q_0 \in L^\infty(\Omega)$ satisfying  $\|q_0\|_{L^{\infty}(\Omega)}\leq R$ and \eqref{hyp:dir}.

There exist $\delta,C > 0$ and $L \in \N$ depending only on $\Omega$, $C_\rho$, $R$, $\W$ and $\|(-\Delta+q_0)^{-1}\|_{H^{-1}(\Omega)\to H^1_0(\Omega)}$ such that for every $q_1,q_2 \in L^{\infty}(\Omega)$ satisfying   \eqref{eq:q1q2}, if
\begin{equation}\label{eq:delta}
\|q_0-q_j\|_{L^2(\Omega)} \leq \delta \qquad j=1,2,
\end{equation}
then $q_1$ and $q_2$ satisfy \eqref{hyp:dir} and
\begin{equation*} 
\|q_2 -q_1\|_{L^2(\Omega)}\leq C\left\|\left(f^L_{n,1}-f^L_{n,2} \right)_{n=1}^N \right\|_{H^{1/2}(\partial \Omega)^N}+ 16 \varepsilon, 
\end{equation*}
where
\begin{equation}\label{eq:nlbd}
f^L_{n,j} = \sum_{l=1}^L \bigl( (S^{q_0}_{\zeta_n} (\Lambda_{q_j}-\Lambda_{q_0})\bigr)^l (f_{q_0}^{n}),\qquad j=1,2,
\end{equation}
and $S^{q_0}_{\zeta_n}$ is the generalized single layer operator corresponding to the Faddeev-Green function and to the potential $q_0$ (see \eqref{eq:S}).
\end{theorem}

We put together several comments on this result.
\begin{itemize}
\item If $\varepsilon = 0$, namely, if the potentials $q_j$ belong exactly to $\W$, Theorem~\ref{thm:stab} yields uniqueness, since if $f^L_{n,1} = f^L_{n,2}$ for $n=1,\dots,N$ we immediately get $q_1 \equiv q_2$.
\item The number $N$ of the  boundary input voltages $\{f^n_{q_0}\}$ is the same in both Theorems~\ref{thm:stab0} and \ref{thm:stab} and it behaves polynomially in the dimension of $\W$ in some explicit examples (see \cite{alberti2018}). This is a much stronger result than what we obtained in \cite{alberti2019infinite}, where the number of measurements needed for Lipschitz stability in EIT was of the order of $\exp(\mathrm{dim}(\W))$. 
\item While in Theorem~\ref{thm:stab0} the stability constant depends on $N$ explicitly  (and so on $\W$), the constant $C$ of Theorem~\ref{thm:stab} is not explicit. This is due to the fact that the constants appearing in Lemma~\ref{lem:Sqbound} and Proposition~\ref{prop:Sqinv} were not made explicit in order to simplify the proofs. It is reasonable to guess, nonetheless, that the constant depends exponentially on $N$ as in Theorem~\ref{thm:stab0}.
\item The functions $f^L_{n,j}$ do not require boundary data depending on the unknown as in Theorem~\ref{thm:stab0}. Starting from $f^n_{q_0}$, which is known by assumption, one constructs $(S_{\zeta_n}^{q_0}(\Lambda_{q_j}-\Lambda_{q_0}))^l f^n_{q_0}$, $l=1,\ldots,L$, in an iterative way. Here we did not make the dependence of $L$ on $N$  explicit. Nonetheless, if the Lipschitz constant $C$ grows exponentially in $N$, it is easy to prove that $L$ grows linearly (or at worst polynomially) with $N$ and not exponentially, making this nonlinear approach stronger than the linearized one of \cite{alberti2019infinite} for EIT.
\item Another strong point of Theorem~\ref{thm:stab} compared to the stability result of \cite{alberti2019infinite} is the dependence with respect to the mismodeling error $\varepsilon$. Here we have a universal constant multiplying $\varepsilon$ while it easy to show that the techniques of \cite{alberti2019infinite} would require the Lipschitz constant $C$ to multiply $\varepsilon$. This means that with a linearized  approach the mismodeling error $\varepsilon$ would be greatly amplified in the reconstruction, while with the present nonlinear method it is not.
\item As a future research direction, it would be interesting to investigate whether methods based on compressed sensing, and in particular on the approach for inverse problems in PDE developed in \cite{alberti2017infinite}, may be used to reduce the number of measurements by exploiting the sparsity of the unknown. It would also be interesting to consider the global problem, namely, whether it is possible to drop assumption \eqref{eq:delta} (by possibly taking additional measurements).
\end{itemize} 

Theorem~\ref{thm:stab} readily yields a similar result for Calder\'on's problem.

\begin{corollary}\label{cor:main}
Take $d\in\{3,4\}$ and let $\Omega \subseteq \mathbb{T}^d$ be a  bounded Lipschitz domain with connected complement, $\W \subseteq L^\infty(\Omega)$ be a finite-dimensional subspace and $N$ be as in Theorem~\ref{thm:stab0}.
Take $R,\lambda>0$ and $\sigma_0 \in W^{2,\infty}(\Omega)$ satisfying \eqref{eq:lambda}, $\|\sigma_0\|_{W^{2,\infty}(\Omega)}\leq R$ and $\sigma_0 = 1$ in a neighborhood of $\partial \Omega$.

There exist $\delta,C > 0$ and $L \in \N$ depending only on $\Omega$, $R$, $C_\rho$, $\lambda$ and $\W$ such that for every $\sigma_1,\sigma_2 \in W^{2,\infty}(\Omega)$ satisfying  

\begin{equation*}
\norm{\sigma_j}_{W^{2,\infty}(\Omega)}\le R\quad\text{and}\quad \norm{P_\W^\perp  \tfrac{\Delta\sqrt{\sigma_j}}{\sqrt{\sigma_j}}}_{L^2(\T^d)}\le \epsilon,\qquad j=1,2,
\end{equation*}
and $\sigma_1 =\sigma_2= 1$ in a neighborhood of $\partial \Omega$,   if
\begin{equation}\label{eq:sigmajsigma0}
\left\|\sigma_j-\sigma_0 \right\|_{H^2(\Omega)} \leq \delta \qquad \text{for } j=1,2,
\end{equation}
then
\begin{equation*} 
\|\sigma_2 -\sigma_1\|_{L^2(\Omega)}\leq C\left\|\left(f^L_{n,1}-f^L_{n,2} \right)_{n=1}^N \right\|_{H^{1/2}(\partial \Omega)^N}+ c(\Omega,\lambda) \varepsilon, 
\end{equation*}
where
\begin{equation}\notag
f^L_{n,j} = \sum_{l=1}^L \bigl( (S^{q_0}_{\zeta_n} (\Lambda_{\sigma_j}-\Lambda_{\sigma_0})\bigr)^l (f_{q_0}^{n}),\qquad q_0 = \frac{\Delta\sqrt{\sigma_0}}{\sqrt{\sigma_0}}.
\end{equation}
\end{corollary}

For the Gel'fand-Calder\'on  problem, the a priori hypothesis on the potential $q$ is written directly in terms of the subspace $\W$, namely, $q$ is assumed to be (almost in) $\W$. Here, however, we assume $\frac{\Delta\sqrt{\sigma}}{\sqrt{\sigma}}$ to be well-approximated by $\W$. This is a shortcoming of the approach, and we do not know whether it is possible to derive a result involving an assumption on $\sigma$ directly, as in \cite{Harrach_2019,alberti2019infinite}. We leave this issue as an interesting  open problem.

\section{Lipschitz stability}\label{sec:uniqueness}

This section contains the proof of Theorem~\ref{thm:stab}. The Lipschitz continuity of the forward map will be a crucial ingredient.
\begin{lemma}\label{lem:lambdif}
Let $\Omega \subseteq \R^d$, $d =3,4$, be an open bounded domain and $q_1,q_2 \in L^{\infty}(\Omega)$ satisfy \eqref{hyp:dir} and $\norm{q_j}_\infty \le R$ for some $R>0$ and
\[
\max\bigl(\|(-\Delta+q_1)^{-1}\|_{H^{-1}(\Omega)\to H^1_0(\Omega)},\|(-\Delta+q_2)^{-1}\|_{H^{-1}(\Omega)\to H^1_0(\Omega)}\bigr)\le U
\]
for some $U>0$. Then
\begin{equation*}
\| \Lambda_{q_1}-\Lambda_{q_2}\|_* \leq c \|q_1-q_2\|_{L^2(\Omega)},
\end{equation*}
where $\| \cdot \|_* = \| \cdot \|_{H^{1/2}(\partial \Omega) \to H^{-1/2}(\partial \Omega)}$ and $c>0$ depends only on   $\Omega$, $R$ and $U$.
\end{lemma}
\begin{proof}
The operator norm is defined as
\[
\| \Lambda_{q_1}-\Lambda_{q_2}\|_* =\sup_{
\substack{
 f_1,f_2 \in H^{\frac 12}(\partial \Omega), \\  \|f_1\|_{H^{\frac 12}}=\|f_2\|_{H^{\frac 12}} =1
}
} |\langle f_1,  (\Lambda_{q_1}-\Lambda_{q_2})f_2\rangle_{H^{\frac 12},H^{-\frac 12}}|.
\]
From Alessandrini's identity \cite{alessandrini1988} we have 
\begin{align*}
|\langle f_1, (\Lambda_{q_1}-\Lambda_{q_2})f_2\rangle_{H^{\frac 12},H^{-\frac 12}}|&= \left|\int_{\Omega}(q_1-q_2)u_1 u_2\right|\\
&\leq \|q_1-q_2\|_{L^2(\Omega)}\|u_1\|_{L^4(\Omega)}\|u_2\|_{L^4(\Omega)},
\end{align*}
by H\"older's inequality and where $u_j \in H^1(\Omega)$ is the unique solution of $(-\Delta +q_j)u_j=0$ in $\Omega$ and $u_j=f_j$ on $\partial \Omega$. In dimension $d=3,4$, by Sobolev embedding and the standard energy estimate for elliptic equations we have 
\[
\|u_j\|_{L^4(\Omega)} \leq c(\Omega) \|u_j\|_{H^1(\Omega)} \leq c(\Omega,R,U) \|f_j\|_{H^{\frac 12}(\partial \Omega)}.
\]
The proof follows.
\end{proof}

The following lemma shows,  by proving a quantitative estimate, that hypothesis \eqref{hyp:dir} is stable under small $L^2$ perturbations of $q$.
\begin{lemma}\label{lem:resolvent}
Let $\Omega \subseteq \R^d$, $d =3,4$, be an open bounded domain and $q_0 \in L^{\infty}(\Omega)$ satisfy \eqref{hyp:dir}. There exists $\delta>0$ depending only on 
$\Omega$ and $\|(-\Delta+q_0)^{-1}\|_{H^{-1}(\Omega)\to H^1_0(\Omega)}$ such that if $q\in L^{\infty}(\Omega)$ satisfies  
\[
\norm{q-q_0}_{L^2(\Omega)}\le \delta,
\]
then $q$ satisfies  \eqref{hyp:dir} and
\[
\|(-\Delta+q)^{-1}\|_{H^{-1}(\Omega)\to H^1_0(\Omega)}\le 2\|(-\Delta+q_0)^{-1}\|_{H^{-1}(\Omega)\to H^1_0(\Omega)}.
\]
\end{lemma}
\begin{proof}
Take $q\in L^{\infty}(\Omega)$ satisfying $\norm{q-q_0}_{L^2(\Omega)}\le \delta$ for some $\delta>0$ to be determined later. 
Let $T=-\Delta+q_0\colon H^1_0(\Omega)\to H^{-1}(\Omega)$ and $M_{q-q_0}\colon H^1_0(\Omega)\to H^{-1}(\Omega)$ be the operator of multiplication by $q-q_0$. For $F\in H^{-1}(\Omega)$ let us consider the Dirichlet problem
\[
(-\Delta + q)u=F
\]
for $u\in H^1_0(\Omega)$. This may be rewritten as
\[
(I_{H^1_0(\Omega)}+T^{-1}M_{q-q_0}) u = T^{-1} F.
\]
In order to conclude, it is enough to show that $\|T^{-1}M_{q-q_0}\|_{H^1_0\to H^1_0}\le\frac12$ for $\delta$ small enough. Observe that
\[
\|T^{-1}M_{q-q_0}\|_{H^1_0\to H^1_0} \le \|T^{-1}\|_{H^{-1}\to H^1_0} \|M_{q-q_0}\|_{H^1_0\to H^{-1}}.
\]
It remains to estimate the last factor. For $v\in H^1_0(\Omega)$, by the Sobolev embedding theorem we have
\[
\begin{split}
\|M_{q-q_0}v\|_{H^{-1}(\Omega)} &=\sup_{w\in H^1_0, \|w\|_{H^1_0}=1}|\int_\Omega (q-q_0)vw\,dx|\\
&\le \sup_{\|w\|_{H^1_0}=1} \|q-q_0\|_{L^2(\Omega)} \|v\|_{L^4(\Omega)} \|w\|_{L^4(\Omega)} \\
&\le c(\Omega) \|q-q_0\|_{L^2(\Omega)}  \sup_{\|w\|_{H^1_0}=1} \|v\|_{H^1_0(\Omega)} \|w\|_{H^1_0(\Omega)}\\
&\le c(\Omega) \delta  \|v\|_{H^1_0(\Omega)}.
\end{split}
\]
This gives $\|M_{q-q_0}\|_{H^1_0\to H^{-1}}\le c(\Omega)\delta$, and so
\[
\|T^{-1}M_{q-q_0}\|_{H^1_0\to H^1_0}\le c(\Omega)  \|T^{-1}\|_{H^{-1}\to H^1_0} \delta \le\frac12,
\]
provided that $\delta=(2c(\Omega)\|T^{-1}\|_{H^{-1}\to H^1_0})^{-1}$.
\end{proof}

We need to introduce the following functions and operators for $\zeta\in \C^d$:
\begin{align}
g_\zeta(x) &= -\left( \frac{1}{2\pi}\right)^d \int_{\R^d}\frac{e^{i\xi \cdot x} }{\xi \cdot \xi + 2 \zeta \cdot \xi}\,d\xi,\notag\\
G_\zeta(x) &= e^{i \zeta \cdot x} g_\zeta(x),\notag\\
G^{q_0}_\zeta(x,y) &= G_\zeta(x-y) + \int_{\R^d}G_\zeta(x-z)q_0(z)G^{q_0}_\zeta(z,y)dz,\notag\\
S^{q_0}_\zeta f(x) &= \int_{\partial \Omega}G^{q_0}_\zeta(x,y) f(y)d\sigma(y).
\label{eq:S}
\end{align}
The function $G_\zeta$ is the Faddeev-Green function, the function $G^{q_0}_\zeta$ was introduced in \cite{novikov2004} (it is called $R^0$ there) and $S^{q_0}_\zeta$ is a generalized single layer operator corresponding to the potential $q_0$. In Lemma~\ref{lem:Sqbound} we prove that $S^{q_0}_\zeta \colon H^{-1/2}(\partial \Omega) \to H^{1/2}(\partial \Omega)$ is bounded for $|\zeta|\ge c_1(R)$.

We are now ready to prove Theorem~\ref{thm:stab}.

\begin{proof}[Proof of Theorem~\ref{thm:stab}]
With an abuse of notation, several different positive constants depending only on $\Omega$, $R$,  $C_\rho$, $\W$ and $\|(-\Delta+q_0)^{-1}\|_{H^{-1}(\Omega)\to H^1_0(\Omega)}$ will be denoted by the same letter $c$ (the dependence on $N$ is omitted since, by Theorem~\ref{thm:stab0}, $N$ depends only on $\Omega$ and $\W$). Take $L\in\N$ and $\delta>0$ (to be determined later), and let $q_1,q_2\in L^\infty(\Omega)$ satisfy \eqref{eq:q1q2} and \eqref{eq:delta}. 

First, note that by Lemma~\ref{lem:resolvent} it is possible to choose $\delta$ small enough (depending only on $\Omega$ and $\|(-\Delta+q_0)^{-1}\|_{H^{-1}(\Omega)\to H^1_0(\Omega)}$ ) so that $q_1$ and $q_2$ satisfy \eqref{hyp:dir} and 
\begin{equation}\label{eq:resolventqj}
\|(-\Delta+q_j)^{-1}\|_{H^{-1}(\Omega)\to H^1_0(\Omega)}\le 2 \|(-\Delta+q_0)^{-1}\|_{H^{-1}(\Omega)\to H^1_0(\Omega)},\qquad j=1,2.
\end{equation}

For $n=1,\dots,N$   set
\[
r_n^L = \left\|f^L_{n,1}-f^L_{n,2} \right\|_{H^{1/2}(\partial \Omega)}.
\] 
Observe that, in view of \eqref{eq:zeta}, \eqref{cond:rho} and \eqref{eq:notation} we have
\begin{equation}\label{eq:D}
|\zeta_n|\le D N,\qquad n=1,\dots,N,
\end{equation}
for some $D>0$ depending only on $R$ and $C_\rho$.

We claim that the following inequality holds:
\begin{equation}\label{est:psidif}
\| f^n_{q_2} - f^n_{q_1}\|_{H^{1/2}(\partial \Omega)}\leq c \frac{L+1}{2^L}\|q_2-q_1\|_{L^2(\Omega)} + 2 r_n^L,\quad n=1,\dots,N.
\end{equation}
Take $n\in\{1,\dots,N\}$. From \cite[Theorem 1]{novikov2004} we have that $f^n_{q_j}$ satisfies the following boundary integral equation:
\begin{equation}\label{eq:bie}
f^n_{q_j} = f^n_{q_0} + S^{q_0}_{\zeta_n} (\Lambda_{q_j}-\Lambda_{q_0})f^n_{q_j},\qquad j=1,2.
\end{equation}
By Lemma \ref{lem:Sqbound} and \eqref{eq:D} we have that the operator $S^{q_0}_{\zeta_n} \colon H^{-1/2}(\partial \Omega) \to H^{1/2}(\partial \Omega)$ is bounded with
\begin{equation}\label{eq:normS}
\norm{S^{q_0}_{\zeta_n}}_{H^{-1/2}\to H^{1/2}}\le c.
\end{equation}
Further, Lemma~\ref{lem:lambdif}, \eqref{eq:delta} and \eqref{eq:resolventqj} yield
\begin{equation}\label{est:wp}
\| \Lambda_{q_j}-\Lambda_{q_0}\|_* \leq c(\Omega,R,\|(-\Delta+q_0)^{-1}\|_{H^{-1}(\Omega)\to H^1_0(\Omega)})\, \delta,
\end{equation}
where $\| \cdot \|_* = \|\cdot \|_{H^{1/2}\to H^{-1/2}}$. Choose $\delta > 0$ small enough (depending only on $\Omega$, $R$, $C_\rho$, $\W$ and $\|(-\Delta+q_0)^{-1}\|_{H^{-1}(\Omega)\to H^1_0(\Omega)}$) so that \eqref{eq:resolventqj} holds and
\begin{equation}\label{eq:12}
\|S^{q_0}_{\zeta_n} (\Lambda_{q_j}-\Lambda_{q_0})\|_{H^{1/2}\to H^{1/2}} \leq \frac{1}{2}.
\end{equation}
Thus, equation \eqref{eq:bie} can be solved with Neumann series converging in $H^{1/2}(\partial \Omega)$:
\begin{equation*}
f^n_{q_j} = \sum_{l= 0}^{+\infty}(S^{q_0}_{\zeta_n} (\Lambda_{q_j}-\Lambda_{q_0}))^l f^n_{q_0}, \quad j=1,2.
\end{equation*}
Then we obtain:
\begin{align*}
f^n_{q_2}-f^n_{q_1} &= \sum_{l \geq L+1} (S^{q_0}_{\zeta_n} (\Lambda_{q_2}-\Lambda_{q_0}))^l f^n_{q_0} -\sum_{l \geq L+1} (S^{q_0}_{\zeta_n} (\Lambda_{q_1}-\Lambda_{q_0}))^l f^n_{q_0} + f_{n,2}^L-f_{n,1}^L\\
&=  (S^{q_0}_{\zeta_n} (\Lambda_{q_2}-\Lambda_{q_0}))^{L+1} f^n_{q_2} - (S^{q_0}_{\zeta_n} (\Lambda_{q_1}-\Lambda_{q_0}))^{L+1} f^n_{q_1} + f_{n,2}^L-f_{n,1}^L\\
&= ((S^{q_0}_{\zeta_n} (\Lambda_{q_2}-\Lambda_{q_0}))^{L+1} - (S^{q_0}_{\zeta_n} (\Lambda_{q_1}-\Lambda_{q_0}))^{L+1})f^n_{q_2}\\
&\quad + (S^{q_0}_{\zeta_n} (\Lambda_{q_1}-\Lambda_{q_0}))^{L+1}(f^n_{q_2}-f^n_{q_1}) + f_{n,2}^L-f_{n,1}^L.
\end{align*}
As a result, by \eqref{eq:12} we have
\begin{align*}
\|f^n_{q_2}-f^n_{q_1}\|_{H^{1/2}}&\leq \|(S^{q_0}_{\zeta_n} (\Lambda_{q_2}-\Lambda_{q_0}))^{L+1} - (S^{q_0}_{\zeta_n} (\Lambda_{q_1}-\Lambda_{q_0}))^{L+1}\|_{H^{1/2}\to H^{1/2}} \|f^n_{q_2}\|_{H^{1/2}}\\
&\quad + \|(S^{q_0}_{\zeta_n} (\Lambda_{q_1}-\Lambda_{q_0}))^{L+1}\|_{H^{1/2}\to H^{1/2}} \|f^n_{q_2}-f^n_{q_1}\|_{H^{1/2}}+r_n^L\\
&\leq c \|(S^{q_0}_{\zeta_n} (\Lambda_{q_2}-\Lambda_{q_0}))^{L+1} - (S^{q_0}_{\zeta_n} (\Lambda_{q_1}-\Lambda_{q_0}))^{L+1}\|_{H^{1/2}\to H^{1/2}}\\
&\quad + \frac{1}{2^{L+1}} \|f^n_{q_2}-f^n_{q_1}\|_{H^{1/2}}+r_n^L,
\end{align*}
where we used the fact that 
\begin{equation}\label{eq:normf}
\|f^n_{q_2}\|_{H^{1/2}}\leq  c
\end{equation}
 (see \eqref{eq:D} and \cite[Proof of Theorem 2]{alberti2018}). As a consequence, we have
\[
\|f^n_{q_2}-f^n_{q_1}\|_{H^{1/2}}\le c\|(S^{q_0}_{\zeta_n} (\Lambda_{q_2}-\Lambda_{q_0}))^{L+1} - (S^{q_0}_{\zeta_n} (\Lambda_{q_1}-\Lambda_{q_0}))^{L+1}\|_{H^{1/2}\to H^{1/2}}+2r_n^L.
\]

In order to estimate the remaining term we need the following identity for bounded linear operators $A,B$ on $H^{1/2}(\partial \Omega)$:
\begin{equation*}
A^{L+1}-B^{L+1} = \sum_{h=0}^L A^h (A-B)B^{L-h}.
\end{equation*}
Putting $A = S^{q_0}_{\zeta_n} (\Lambda_{q_2}-\Lambda_{q_0})$ and $B= S^{q_0}_{\zeta_n} (\Lambda_{q_1}-\Lambda_{q_0})$  we find
\begin{equation}\label{eq:A-B}
\begin{aligned}
&\|(S^{q_0}_{\zeta_n} (\Lambda_{q_2}-\Lambda_{q_0}))^{L+1} - (S^{q_0}_{\zeta_n} (\Lambda_{q_1}-\Lambda_{q_0}))^{L+1}\|_{H^{1/2}\to H^{1/2}} \\
&\qquad \leq \|S^{q_0}_{\zeta_n} (\Lambda_{q_2}-\Lambda_{q_1})\|\sum_{h=0}^L \|S^{q_0}_{\zeta_n} (\Lambda_{q_2}-\Lambda_{q_0})\|^h \|S^{q_0}_{\zeta_n} (\Lambda_{q_1}-\Lambda_{q_0})\|^{L-h}\\
&\qquad \leq c\frac{L+1}{2^L}\|\Lambda_{q_2}-\Lambda_{q_1}\|_*,
\end{aligned}
\end{equation}
where we used \eqref{eq:normS} and \eqref{eq:12}.
Using Lemma~\ref{lem:lambdif} again we have $\|\Lambda_{q_2}-\Lambda_{q_1}\|_* \leq c  \| q_2 - q_1\|_{L^2(\Omega)}$. So estimate \eqref{est:psidif} is proved.\bigskip

In order to finish the proof we need to connect $(\Lambda_{q_1} - \Lambda_{q_2})(f^n_{q_1})$ to $f^n_{q_2}-f^n_{q_1}$. This is done again via a boundary integral equation from \cite[Theorem 1]{novikov2004}:
\begin{equation}\label{eq:intgralq1q2}
f^n_{q_1}- f^n_{q_2} =  -S^{q_2}_{\zeta_n} (\Lambda_{q_2}-\Lambda_{q_1})f^n_{q_1},\qquad n=1,\dots,N.
\end{equation}
From Proposition \ref{prop:Sqinv}, \eqref{eq:resolventqj} and \eqref{eq:D}, we immediately obtain the following estimate:
\begin{equation*}
\|(\Lambda_{q_1} - \Lambda_{q_2})f^n_{q_1}\|_{H^{-1/2}(\partial \Omega)}\leq c \|f^n_{q_2}-f^n_{q_1}\|_{H^{1/2}(\partial \Omega)} ,\qquad n=1,\dots,N. 
\end{equation*}
Therefore, by \eqref{est:psidif} and Theorem~\ref{thm:stab0} we obtain
\begin{align*}
\|q_2 - q_1\|_{L^2(\Omega)} &\leq   e^{c(\Omega,R) N}
\norm{\left((\Lambda_{q_2} - \Lambda_{q_1})f^n_{q_1}\right)_{n=1}^N}_{H^{-1/2}(\partial \Omega)^N}+8\varepsilon \\
&\leq  c \norm{\left( f^n_{q_2}-f^n_{q_1}\right)_{n=1}^N}_{H^{1/2}(\partial \Omega)^N}+8\varepsilon\\
&\leq c \left(\frac{L+1}{2^L}\|q_2 -q_1\|_{L^2(\Omega)} + \|(r_n^L)_n\|_2 \right)  +8\varepsilon. 
\end{align*}
Taking $L$ sufficiently large, so that $ c\,\frac{L+1}{2^L} \leq \frac 1 2$, we obtain the Lipschitz stability estimate of the statement.
\end{proof}

We conclude this section by showing how Corollary~\ref{cor:main} on the Calder\'on problem is an immediate consequence of Theorem~\ref{thm:stab}.

\begin{proof}[Proof of Corollary~\ref{cor:main}]
Set $q_j = \frac{\Delta\sqrt{\sigma_j}}{\sqrt{\sigma_j}}$ for $j=0,1,2$. Thanks to \eqref{eq:lambda}, by assumption we have
\[
\norm{q_j}_\infty \le c(R,\lambda),\qquad j=0,1,2.
\]
Similarly, 
\eqref{eq:sigmajsigma0} and \eqref{eq:lambda} yield
\[
\left\| q_0-q_j\right\|_{L^2(\Omega)} \leq c(R,\lambda) \delta, \qquad  j=1,2.
\]
Thus, it is possible to apply Theorem~\ref{thm:stab} to $q_0$, $q_1$ and $q_2$, by using the standard Liouville transformation. For the details, the reader is referred to  \cite[Corollaries~1 and 2]{alberti2018}.
\end{proof}

\section{Reconstruction}\label{sec:recon}

We extend the reconstruction scheme of our previous work \cite{alberti2018} to the setting of the present article. We also incorporate noise and mismodeling errors. We consider only Gel'fand-Calder\'on's problem; the reconstruction algorithm for Calder\'on's problem may be obtained by using the usual change of variables $q=\frac{\Delta\sqrt{\sigma}}{\sqrt{\sigma}}$, as in Corollary~\ref{cor:main}. Here we do not assume to know the boundary traces of the CGO solutions of an unknown potential $\bar q$, but only those of an approximation $q_0$.

Consider the setting of Theorem~\ref{thm:stab} and assume $\|q_0 - \bar q\|_{L^2(\Omega)} \leq \delta$. We now present a reconstruction algorithm able to recover $\bar q$ from $q_0$ and a finite number of measurements in the same spirit of Theorem~\ref{thm:stab}. The  noisy boundary measurements correspond to finitely many evaluations of the boundary map
\[
\Lanoise = \Lambda_{\bar q}+E,
\]
where $E\colon H^{\frac12}(\partial\Omega)\to H^{-\frac12}(\partial\Omega)$ is a linear operator representing  noise ($E$ stands for \textit{error}) and satisfies
\[
\norm{E}_*\le\eta,
\]
where $\eta\ge 0$ is the noise level.

From $q_0$ we can stably compute its associated CGO solution $f^n_{q_0}$, for $n=1,\dots,N$, and the quantity
\begin{equation*}
f^L_{n} = \sum_{l=0}^L  (S^{q_0}_{\zeta_n} (\Lanoise-\Lambda_{q_0}))^l f^n_{q_0},
\end{equation*}
which can be obtained iteratively by solving Dirichlet problems for the Schr\"odinger equation \eqref{eq:schr} with boundary values $(S^{q_0}_{\zeta_n} (\Lanoise-\Lambda_{q_0}))^l f^n_{q_0}$ for $l=0,\dots,L-1$.

Let $L^\infty_R(\Omega) = \{ q \in L^\infty(\Omega):\|q\|_\infty \leq R\}$ be equipped with the distance induced by the $L^2$ norm.
We define the nonlinear mapping $A\colon L^\infty_R(\Omega) \to \W_R$ by
\begin{equation}\label{eq:A}
A(q) = P_{\W_R}( F^{-1}P_N T(q) + F^{-1}P_N^\perp Fi(q) - F^{-1}P_N B  (q)),
\end{equation}
where $i$ is the extension operator already defined and, as in \cite{alberti2018}:
\begin{itemize}
\item $\W_R=L^\infty_R(\Omega)\cap\W$;
\item $F\colon \H \to \ell^2$ is the  discrete Fourier transform defined by
\begin{equation}\label{def:DFT}
(F q)_n=\int_{\T^d} q(x)e^{-2\pi ik_n\cdot x}\,dx,\qquad n\in\N;
\end{equation}
\item $B\colon L^2(\Omega) \to \ell^2$ is the  perturbation given by
\begin{equation*}%\label{def:DFT}
\begin{split}
(B (q))_n&=\int_{\Omega} q(x)e^{-2\pi ik_l\cdot x}r_q^{k_n,t_n}(x)\,dx\\
 &=\langle  e^{\tilde\zeta_n\cdot x}, (\Lambda_q - \Lambda_{0})f^n_q \rangle_{{ H^{\frac12}(\partial \Omega) \times  H^{-\frac12}(\partial \Omega)}}-(Fq)_n,
 \end{split}
\end{equation*}
where, as in \eqref{eq:zeta}, $\tilde\zeta_n = -i(\pi k_n - t_n\xi) - \sqrt{t_n^2 +\pi^2|k_n|^2}\eta$ (note that the second identity holds only when $q$ satisfies \eqref{hyp:dir});
\item $P_N\colon\ell^\infty\to\ell^\infty$ is the projection onto the first $N$ components, namely $P_N(a_1,a_2,\dots)=(a_1,\dots,a_N,0,0,\dots)$, and $P_N^\perp=I-P_N$;
\item $P_{\W_R}$ is the projection from $L^2(\T^d)$ onto the closed and convex set $i(\W_R)$;
\item and
\begin{equation*}
(T(q))_n = \left\langle  e^{\tilde\zeta_n\cdot x}, (\Lanoise - \Lambda_{0}) \left( f^L_n+(S^{q_0}_{\zeta_n} (\Lanoise-\Lambda_{q_0}))^{L+1}  f^{n}_q\right) \right\rangle_{{ H^{\frac12}(\partial \Omega) \times  H^{-\frac12}(\partial \Omega)}}.
\end{equation*}
\end{itemize}

The main result of this section reads as follows.
\begin{theorem}\label{theo:rec}
Take $d\in\{3,4\}$ and $R,\epsilon>0$ and let $\Omega \subseteq \mathbb{T}^d$ be a  bounded Lipschitz domain with connected complement, $\W \subseteq L^\infty(\Omega)$ be a finite-dimensional subspace and $N$ and $\delta$ be as in Theorem~\ref{thm:stab}.
Take  $q_0 \in L^\infty_R(\Omega)$ satisfying \eqref{hyp:dir} and $\bar q\in L^\infty_R(\Omega)$ satisfying
\begin{equation*}
\|\bar q -P_{\W_R} (\bar q)\|_{L^2(\T^d)}\leq \varepsilon,\qquad \|q_0-\bar q\|_{L^2(\Omega)} \leq \delta.
\end{equation*}

There exist  $L \in \N$ and $C,S>0$ depending only on $\Omega$, $C_\rho$, $R$, $\W$ and $\|(-\Delta+q_0)^{-1}\|_{H^{-1}(\Omega)\to H^1_0(\Omega)}$ such that, if $\eta\in [0,S]$ and $q^1 \in \W_R$ is any initial guess, then the sequence 
\[
q^n = A(q^{n-1}), \qquad  n \geq 2,
\]
converges to $\q \in \W_R$ and
\begin{equation}\label{eq:banach}
\|\q- q^n\|_{L^2(\Omega)} \leq 8 \left(\frac 7 8\right)^n\|q^2 - q^1\|_{L^2(\Omega)},\qquad n\ge 1.
\end{equation}
Further, we have
\begin{equation}\label{eq:q*qbar}
\|\bar q - \q\|_{L^2(\Omega)} \leq 14\epsilon+C\eta.
\end{equation}
\end{theorem}
\begin{remark*}
As expected from Theorem~\ref{thm:stab}, in absence of noise ($\eta=0$) and without modeling errors ($\bar q\in \W$), the unknown $\bar q$ may be recovered exactly, as the limit $\bar q=\lim_n q^n$. Further, it is worth observing that the stability with respect to noise in the data  and with respect to modeling errors given in \eqref{eq:q*qbar} is consistent with the estimate of Theorem~\ref{thm:stab}: the factor $\epsilon$ is multiplied by an absolute constant, while the noise level $\eta$ by a constant that becomes larger as $\dim\W$ increases. 
\end{remark*}
\begin{remark*}
This result may be seen as a first step towards a new regularization strategy for Gel'fand-Calder\'on's and Calder\'on's problems \cite{rondi2008,knudsen2009,rondi2016}, by considering an exhaustive sequence of nested subspaces $\W_m$. In this case the method would fall into the classes of regularization by projection \cite{engl1996regularization} and regularization by discretization \cite{scherzer_book}.
\end{remark*}

\begin{proof}
With an abuse of notation, several different positive constants depending only on $\Omega$, $R$,  $C_\rho$, $\W$ and $\|(-\Delta+q_0)^{-1}\|_{H^{-1}(\Omega)\to H^1_0(\Omega)}$ will be denoted by the same letter $c$. The proof is divided into four steps.

\textit{Step 1: $A$ is Lipschitz continuous.} In view of \cite[Lemma~1]{alberti2018} we have that $B$ is a contraction on $L^\infty_R(\Omega)$, namely
\begin{equation}\label{eq:Bcontraction}
\left\|B(q_{2})-B(q_{1})\right\|_{\ell^{2}} \leqslant \frac{1}{2}\left\|q_{2}-q_{1}\right\|_{L^{2}\left(\Omega\right)},\qquad q_1,q_2\in L^\infty_R(\Omega).
\end{equation}
Further, thanks to Parseval's identity the map $F$ is an isometry, and so for $q_1,q_2\in L^\infty_R(\Omega)$ we have
\begin{equation}\label{eq:AT}
\norm{A(q_2)-A(q_1)}_{L^2(\T^d)}\le \norm{P_NT(q_2)-P_NT(q_1)}_{\ell^2}+\frac32\norm{q_2-q_1}_{L^2(\Omega)},
\end{equation}
where we also used that $P_{\W_R}$ is Lipschitz continuous with constant $1$ by the Hilbert projection theorem. 
It remains to estimate the term with $P_NT$.

For $n\in\N$ we have
\[
(T(q_2) - T(q_1))_n = \left\langle  e^{\tilde\zeta_n\cdot x}, (\Lanoise - \Lambda_{0}) (S^{q_0}_{\zeta_n} (\Lanoise-\Lambda_{q_0}))^{L+1}  (f^{n}_{q_2}- f^{n}_{q_1}) \right\rangle_{{ H^{\frac12} \times  H^{-\frac12}}},
\]
which gives for $n\in\{1,\dots,N\}$
\begin{align*}
|(T(q_2) - T(q_1))_n| &\leq  e^{c|\tilde \zeta_n|} \|(\Lanoise - \Lambda_{0}) (S^{q_0}_{\zeta_n} (\Lanoise-\Lambda_{q_0}))^{L+1}  (f^{n}_{q_2}- f^{n}_{q_1})\|_{H^{\frac 12}(\partial \Omega)}\\
&\leq c \, \| S^{q_0}_{\zeta_n} (\Lambda_{\bar q}-\Lambda_{q_0}+E)\|^{L+1}_{H^{\frac12}\to H^{\frac12}}  \|f^{n}_{q_2}- f^{n}_{q_1}\|_{H^{\frac 12}(\partial \Omega)}\\
&\leq c \,(1/2+c\eta)^{L+1}  \|f^{n}_{q_2}- f^{n}_{q_1}\|_{H^{\frac 12}(\partial \Omega)},
\end{align*}
where we used \eqref{eq:normS}, \eqref{eq:12} and  Lemma~\ref{lem:lambdif}. Choose $S=\frac{1}{4c}$, so that $c\eta\le\frac14$ for $\eta\le S$. Thus, by \eqref{eq:normS},  \eqref{eq:normf}, \eqref{eq:intgralq1q2} and  Lemma~\ref{lem:lambdif}, we have
\begin{align*}
|(T(q_2) - T(q_1))_n|&\leq c\,(3/4)^{L+1}\left\|S^{q_1}_{\zeta_n}(\Lambda_{q_1}-\Lambda_{q_2})f_{q_2}^n \right\|_{H^{\frac 12}(\partial \Omega)}\\
&\leq c\,(3/4)^{L+1}\|q_2-q_1\|_{L^2(\Omega)}.
\end{align*}
As a result, we have
$
\|P_N(T(q_2)-T(q_1))\|_{\ell^2}\le c\,(3/4)^{L+1}\|q_2-q_1\|_{L^2(\Omega)}.
$
Choose $L$ sufficiently large so that $c\,(3/4)^{L+1} \leq \frac 18$ (the constant $\frac18$ will be handy below). Then
\begin{equation}\label{eq:PN}
\|P_N(T(q_2)-T(q_1))\|_{\ell^2}\le \frac18\|q_2-q_1\|_{L^2(\Omega)},\qquad q_1,q_2 \in L^\infty_R(\Omega).
\end{equation}
(From now on, $L$ is fixed). Thus,  by \eqref{eq:AT} we obtain
\begin{equation}\label{eq:Lip}
\| A(q_2) - A(q_1)\|_{L^2(\T^d)} \leq \frac{13}{8}\|q_2 -q_1\|_{L^2(\Omega)}, \quad q_1,q_2 \in L^\infty_R(\Omega).
\end{equation}

\textit{Step 2: $A|_{\W_R}$ is a contraction and has a fixed point.} The number of measurements $N$, which is given in Theorem~\ref{thm:stab0}, is chosen so that
$
\left\|P_{N}^{\perp} F P_{\mathcal{W}}\right\|_{L^{2}\left(\mathbb{T}^{d}\right) \rightarrow \ell^{2}} \leqslant \frac{1}{4}
$ \cite{alberti2018}.
Thus, \eqref{eq:A}, \eqref{eq:Bcontraction} and \eqref{eq:PN} yield
\begin{equation}\label{eq:contraction}
\| A(q_2) - A(q_1)\|_{L^2(\T^d)} \leq \frac{7}{8}\|q_2 -q_1\|_{L^2(\Omega)}, \quad q_1,q_2 \in \W_R.
\end{equation}
Note that $\W_R$ is a complete metric space with the distance given by the $L^2$ norm. As a consequence, the Banach fixed-point theorem yields the existence of a fixed point $\q$, and \eqref{eq:banach} holds. It remains to prove \eqref{eq:q*qbar}.

\textit{Step 3: $\norm{A(\bar q)-P_{\W_R}(\bar q)}_{L^2(\T^d)}\le c\eta$.} Let us consider the map $\tilde A\colon L^\infty_R(\Omega) \to \W_R$ corresponding to the noiseless case. Namely, we define
\begin{equation*}
\tilde A(q) = P_{\W_R}( F^{-1}P_N \tilde T(q) + F^{-1}P_N^\perp Fi(q) - F^{-1}P_N B  (q)),
\end{equation*}
where
\begin{equation*}
(\tilde T(q))_n = \left\langle  e^{\tilde\zeta_n\cdot x}, (\Lambda_{\bar q} - \Lambda_{0}) \left( \tilde f^L_n+(S^{q_0}_{\zeta_n} (\Lambda_{\bar q}-\Lambda_{q_0}))^{L+1}  f^{n}_q\right) \right\rangle_{{ H^{\frac12}(\partial \Omega) \times  H^{-\frac12}(\partial \Omega)}},
\end{equation*}
and 
\begin{equation*}
\tilde f^L_{n} = \sum_{l=0}^L  (S^{q_0}_{\zeta_n} (\Lambda_{\bar q} -\Lambda_{q_0}))^l f^n_{q_0}.
\end{equation*}

We start by noting that $\tilde A(\bar q)=P_{\W_R}\bar q$. Indeed, observe that $f_{\bar q}^n$ solves the boundary integral equation
\[
f_{\bar q}^n = \tilde f^L_n+(S^{q_0}_{\zeta_n} (\Lambda_{\bar q}-\Lambda_{q_0}))^{L+1}  f^{n}_{\bar q},
\]
which follows immediately from \eqref{eq:bie}. Thus $\tilde T(\bar q) = F\bar q + B(\bar q)$ and so
\[
\begin{split}
\tilde A(\bar q)&= P_{\W_R}( F^{-1}P_N (F\bar q + B(\bar q)) + F^{-1}P_N^\perp F\bar q - F^{-1}P_N B  (\bar q)), \\
&= P_{\W_R}(  F^{-1}P_N F\bar  q+ F^{-1}P_N^\perp F\bar q )\\
&=P_{\W_R}(\bar q).
\end{split}
\]

Since $P_{\W_R}$ is non-expansive and $F$ is an isometry, we have
\[
\norm{A(\bar q)-P_{\W_R}(\bar q)}_{L^2(\T^d)}=\|A(\bar q)-\tilde A(\bar q)\|_{L^2(\T^d)}\le \|P_N\tilde T(\bar q)-P_NT(\bar q)\|_{\ell^2}.
\]
Setting
\[
a_n=\tilde f^L_n+(S^{q_0}_{\zeta_n} (\Lambda_{\bar q}-\Lambda_{q_0}))^{L+1}  f^{n}_{\bar q},\qquad b_n=f^L_n+(S^{q_0}_{\zeta_n} (\Lanoise-\Lambda_{q_0}))^{L+1}  f^{n}_{\bar q},
\]
by using again that $|\tilde\zeta_n|\le c$ and the triangle inequality, we readily derive
\[
\begin{split}
|(\tilde T(\bar q)-T(\bar q))_n| &= |\langle  e^{\tilde\zeta_n\cdot x}, (\Lambda_{\bar q} - \Lambda_{0}) ( a_n ) -
 (\Lanoise - \Lambda_{0}) (  b_n )
 \rangle_{{ H^{\frac12} \times  H^{-\frac12}}}|\\
 &\le c \|(\Lambda_{\bar q} - \Lambda_{0}) ( a_n ) -
 (\Lanoise - \Lambda_{0}) (  b_n )\|_{H^{-\frac12}}\\
  &\le c \|(\Lambda_{\bar q} - \Lambda_{0}) ( a_n-b_n )\|_{H^{-\frac12}} +
c\| (\Lanoise - \Lambda_{\bar q}) (  b_n )\|_{H^{-\frac12}}\\
&\le c\|\Lambda_{\bar q} - \Lambda_{0}\|_*\|a_n-b_n\|_{H^{\frac12}}
+c\| E \|_*\|b_n\|_{H^{\frac12}}\\
&\le c \|a_n-b_n\|_{H^{\frac12}}
+c\eta,
 \end{split}
\]
where the last inequality follows from \eqref{eq:normS},  \eqref{eq:normf} and  Lemma~\ref{lem:lambdif}. It remains to estimate $\|a_n-b_n\|_{H^{\frac12}}$. Using again \eqref{eq:normf} we obtain
\[
\begin{split}
\|a_n-b_n\|_{H^{\frac12}}&\le \|
\tilde f^L_n-f^L_n\|_{H^{\frac12}}+\|\bigl((S^{q_0}_{\zeta_n} (\Lambda_{\bar q}-\Lambda_{q_0}))^{L+1}  
-(S^{q_0}_{\zeta_n} (\Lanoise-\Lambda_{q_0}))^{L+1}\bigr)  f^{n}_{\bar q}\|_{H^{\frac12}}\\
&\le c \sum_{l=0}^{L+1}  \|(S^{q_0}_{\zeta_n} (\Lambda_{\bar q} -\Lambda_{q_0}))^l -(S^{q_0}_{\zeta_n} (\Lanoise -\Lambda_{q_0}))^l \|_{H^{\frac12}\to H^{\frac12}}.
\end{split}
\]
Arguing as in \eqref{eq:A-B}, we can bound the last term with $c\|\Lambda_{\bar q}-\Lanoise\|_*\le c\eta$, so that $\|a_n-b_n\|_{H^{\frac12}}\le c\eta$. Altogether, we have
\begin{equation}\label{eq:boring}
\norm{A(\bar q)-P_{\W_R}(\bar q)}_{L^2(\T^d)}\le \|P_N\tilde T(\bar q)-P_NT(\bar q)\|_{\ell^2}\le c\eta,
\end{equation}
as desired.

\textit{Step 4: Proof of \eqref{eq:q*qbar}.}  Since $\q$ is a fixed point of $A$, we have
\begin{multline*}
\norm{\q-P_{\W_R}(\bar q)}_{L^2(\T^d)}\le \norm{A(\q)-A(P_{\W_R}(\bar q))}_{L^2(\T^d)}\\ +\norm{A(P_{\W_R}(\bar q))-A(\bar q)}_{L^2(\T^d)}+\norm{A(\bar q)-P_{\W_R}(\bar q)}_{L^2(\T^d)}.
\end{multline*}
Thus, by \eqref{eq:Lip} , \eqref{eq:contraction} and \eqref{eq:boring}
we obtain
\[
\begin{split}
\norm{\q-P_{\W_R}(\bar q)}_{L^2(\T^d)}&\le\frac78\norm{\q-P_{\W_R}(\bar q)}_{L^2(\T^d)}+\frac{13}{8}\norm{P_{\W_R}(\bar q)-\bar q}_{L^2(\T^d)}+c\eta
\\&\le\frac78\norm{\q-P_{\W_R}(\bar q)}_{L^2(\T^d)}+\frac{13}{8}\epsilon+c\eta,
\end{split}
\]
so that
$
\norm{\q-P_{\W_R}(\bar q)}_{L^2(\T^d)}\le 13\epsilon+c\eta.
$
Finally, we have
\[
\norm{\q-\bar q}_{L^2(\T^d)}\le \norm{\q-P_{\W_R}(\bar q)}_{L^2(\T^d)}+\norm{P_{\W_R}(\bar q)-\bar q}_{L^2(\T^d)}\le 14\epsilon+c\eta.
\]
This concludes the proof.
\end{proof}

\section{Layer potentials estimates and invertibility properties}\label{sec:layer}

This section is devoted to the proof of new properties of the generalized layer potential.

Throughout this section, we let $\Omega \subseteq \R^d$, $d\geq 2$ be an open bounded domain with Lipschitz boundary, and $q \in L^{\infty}(\R^d)$ be a potential satisfying \eqref{hyp:dir} and such that $\supp (q) \subseteq \Omega$, $\|q\|_{L^{\infty}(\Omega)}\leq R$, for some $R >0$.

We recall from the previous section the following functions:
\begin{align}
g_\zeta(x) &= \left( \frac{1}{2\pi}\right)^d \int_{\R^d}\frac{e^{i\xi \cdot x} }{\xi \cdot \xi + 2 \zeta \cdot \xi}\,d\xi,\notag\\
G_\zeta(x) &= e^{i \zeta \cdot x} g_\zeta(x),\notag \\ \label{def:Gq}
G^q_\zeta(x,y) &= G_\zeta(x-y) - \int_{\R^d}G_\zeta(x-z)q(z)G^q_\zeta(z,y)dz,\quad x,y \in \R^d, x\neq y,\\ \label{def:gensin}
S^q_\zeta f(x) &= \int_{\partial \Omega}G^q_\zeta(x,y) f(y)d\sigma(y),\quad x \in \R^d,\\ \label{def:gq}
 g^q_\zeta(x,y) &= e^{-i \zeta \cdot (x-y)}G^q_\zeta(x,y),\quad x,y \in \R^d, x\neq y.
\end{align}
Note that
\begin{equation*}
-\Delta G_\zeta(x-y) = (-\Delta + q(x))G^q_\zeta(x,y) = \delta(x-y).
\end{equation*}
 We also introduce the generalized double layer potential
\begin{equation*}
D^q_\zeta f(x) = \int_{\partial \Omega}\frac{\partial G^q_\zeta}{\partial \nu_y}(x,y)f(y) d\sigma(y), \quad x \in \R^d \setminus \partial \Omega,
\end{equation*}
and the generalized boundary, or trace, double layer potential by
\begin{equation*}
B^q_\zeta f(x) = \pv \int_{\partial \Omega}\frac{\partial G^q_\zeta}{\partial \nu_y}(x,y)f(y) d\sigma(y), \quad x \in \partial \Omega.
\end{equation*}
Since the singularity of $G^q_\zeta ( x,y)$ for $x$ near $y$ is the same as that of $G_\zeta(x,y)$ (see \cite[Theorem 7.1]{littman1963}), it is locally integrable on $\partial \Omega$ and the trace single layer potential is given by \eqref{def:gensin}. Finally, let us consider the operator $G^q_\zeta$ defined by
\begin{equation*}
(G^q_\zeta f)(x) = \int_\Omega G^q_\zeta(x,y)f(y)dy.
\end{equation*}
We start by showing that $S^q_\zeta$ is bounded.
\begin{lemma}\label{lem:Sqbound}
Let $D>0$ and $\tilde \Omega$ be a bounded $C^{1,1}$ neighborhood  of $\Omega$. There exists $ c_1=c_1(R)>0$ such that, for every $c_1\le |\zeta| \le D$, we have \begin{equation}\label{eq:GqH2}
\|G^q_\zeta f\|_{H^2(\tilde \Omega)} \leq C(D,R, \tilde \Omega) \|f\|_{L^2(\Omega)},\qquad f\in L^2(\Omega),
\end{equation}
and
\begin{equation}\label{eq:Sqbound}
\| S^q_\zeta f\|_{H^{1/2}(\partial \Omega)}\leq C(D,R,  \Omega) \|f\|_{H^{-1/2}(\partial \Omega)},\qquad f\in H^{-1/2}(\partial \Omega).
\end{equation}
\end{lemma}
\begin{remark}
The constants $C$ can be estimated using similar ideas as in \cite[Lemma~2.2]{knudsen2009}, and they grow exponentially in $D$.
\end{remark}

\begin{proof}
In the proof, with an abuse of notation, several different positive absolute constants will be denoted by the same letter $c$. We first prove \eqref{eq:GqH2} and then \eqref{eq:Sqbound}.

\textit{Proof of \eqref{eq:GqH2}.}
For $\delta=\frac34$ (the argument below works for any $\delta\in (\frac12,1)$, but for our purposes it is enough to let $\delta=\frac34$), consider the Hilbert space 
\[
L^2_\delta(\R^d) =\left\{ f : \| f\|_\delta = \left( \int_{\R^d} (1+|x|^2)^\delta |f(x)|^2\right)^{1/2}<+\infty \right\}.
\]
From \cite[Proposition 2.1.b)]{nachman1988} we have that, for $\zeta \in \C^d \setminus \R^d$ with $|\zeta| \geq 1$, 
\begin{equation}\label{est:Gzeta}
\|g_\zeta \ast f\|_{-\delta} \leq \frac{c}{|\zeta|}\|f\|_\delta, \qquad f\in L^2_\delta(\R^d),
\end{equation}
for an absolute constant $c>0$.
Now, let $g^q_\zeta$ be the operator defined by
\begin{equation*}
(g^q_\zeta f)(x) = \int_{\R^d} g^q_\zeta(x,y)f(y)dy,\qquad f\in L^2_\delta(\R^d).
\end{equation*}
We want to extend \eqref{est:Gzeta} to $g^q_\zeta f$. From \eqref{def:gq} and the integral equation \eqref{def:Gq} we have
\begin{align*}
g_\zeta^q f(x) &= g_\zeta * f(x)- \int_{\R^d} g_\zeta (x-z)q(z)(g^q_\zeta f)(z)dz\\
&= g_\zeta * f(x)-(g_\zeta * \textbf{q} (g^q_\zeta f)) (x),
\end{align*}
where $\textbf{q}$ denotes the operator of multiplication by $q$, which maps $L^2_{-\delta}$ to $L^2_\delta$ with norm bounded by $\|q(x)(1+|x|^2)^\delta\|_{L^\infty(\R^d)}$. For $|\zeta| \geq c\|q(x)(1+|x|^2)^\delta\|_{L^\infty(\R^d)}$ the operator $g_\zeta * \textbf{q}$ maps $L^2_{-\delta}$ into itself and  $\|g_\zeta * \textbf{q}\|_{L^2_{-\delta}\to L^2_{-\delta}}\le 1/2$ thanks to \eqref{est:Gzeta} (see \cite[Corollary 2.2]{nachman1988} for more details). Thus
\begin{align*}
g_\zeta^q f = (I + g_\zeta * \textbf{q})^{-1}(g_\zeta * f),
\end{align*}
and we obtain, using \eqref{est:Gzeta},
\begin{equation*}
\|g^q_\zeta f\|_{-\delta} \leq \frac{c}{|\zeta|}\|f\|_\delta,\qquad f\in L^2_\delta(\R^d),
\end{equation*}
for $|\zeta|\ge c(R)$.
Consider now $H^2_\delta(\R^d) = \{f\colon D^\alpha f \in L^2_\delta(\R^d), 0 \leq |\alpha| \leq 2\}$, the weighted Sobolev space with norm
\[
\|f\|_{2,\delta} = \left(\sum_{|\alpha|\leq 2} \|D^\alpha f\|_\delta^2\right)^{1/2}.
\]
Using the same ideas as in the proof of \cite[Lemma~2.11]{nachman1988}, based on \cite{lavine1981}, we obtain
\begin{equation}\label{eq:gqH2}
\|g^q_\zeta f\|_{2,-\delta} \leq c(D,R) \|f\|_\delta, \qquad f\in L^2_\delta(\R^d).
\end{equation}
Then the main estimate \eqref{eq:GqH2} is a direct consequence of \eqref{eq:gqH2}, \eqref{def:gq} and the boundedness of $\tilde \Omega$ and $\Omega$, since
\[
G^q_\zeta f(x)=e^{i\zeta\cdot x}\int_\Omega g^q_\zeta(x,y) e^{-i\zeta\cdot y} f(y)\,dy=e^{i\zeta\cdot x} g^q_\zeta(e^{-i\zeta\cdot y} f)(x).
\]

\textit{Proof of \eqref{eq:Sqbound}.}
Using similar arguments as in \cite[Section 6]{nachman1988} we can rewrite equation \eqref{def:Gq} as follows:
\begin{equation}\label{id:ggq}
G^q_\zeta(x,y) = G_\zeta(x-y) - \int_{\R^d}G^q_\zeta(x,z)q(z)G_\zeta(z-y)dz,
\end{equation}
which yields the identity
\begin{equation}\label{id:ssq}
S^q_\zeta f(x) - S_\zeta f(x) =- \int_{\R^d}G^q_\zeta(x,z)q(z)S_\zeta f(z)dz,
\end{equation}
where $S_\zeta$ is the generalized single layer potential for $q = 0$. We recall \cite[Lemma 2.3]{nachman1988}, which states, for $0 \leq s \leq 1$,
\begin{equation*}
\| S_\zeta f\|_{H^{s+1}(\partial \Omega)}\leq c(D,s,\Omega) \|f\|_{H^s(\partial \Omega)},
\end{equation*}
for $\partial \Omega \in C^{1,1}$. This can be easily extended to $-1\leq s \leq 0$ and $\partial \Omega$ Lipschitz using the same arguments as in the proof of \cite[Lemma 7.1]{Nachman1996}.
Therefore, by the trace theorem and \eqref{eq:GqH2}, we have
\begin{align*}
\| S^q_\zeta f\|_{H^{1/2}(\partial \Omega)}&\leq c(\Omega,D) \|f\|_{H^{-1/2}(\partial \Omega)} +c(\Omega) \left\|\int_{\R^d}G^q_\zeta(\cdot,z)q(z)S_\zeta f(z)dz\right\|_{H^1(\Omega)}\\
&\leq c(D,R,\Omega)\left( \|f\|_{H^{-1/2}(\partial \Omega)} + \| S_\zeta f\|_{L^2(\Omega)}\right)\\
&\leq c(D,R,\Omega)\|f\|_{H^{-1/2}(\partial \Omega)}.
\end{align*}
Here we have used the fact that $u := S_\zeta f$ solves $\Delta u= 0$ in $\Omega$ (see \cite[Lemma~2.4]{nachman1988}) so, by interior regularity one has $\| S_\zeta f\|_{H^1(\Omega)}\leq c(\Omega) \| S_\zeta f\|_{H^{1/2}(\partial \Omega)}$ (see \cite[Theorem~3]{savare1998}).
\end{proof}

In order to study the invertibility of $S^q_\zeta$, we need other technical results. We start with the following solvability result for an exterior Dirichlet problem.
Let $\rho_0 >0$ be such that $\Omega \subseteq B_{\rho_0} = \{ x \in \R^d: |x| < \rho_0\}$. For $\rho > \rho_0$ let $\Omega'_\rho = B_\rho \setminus  \overline\Omega$.

\begin{lemma}\label{lem:ext}
Suppose that  $\Omega' = \R^d \setminus \overline \Omega$ is connected and let $\zeta\in\C^d$ be such that $|\zeta|\ge c_1$, where $c_1$ is given by Lemma~\ref{lem:Sqbound}. For any $f \in H^{1/2}(\partial \Omega)$ there is a unique solution $u$ to the exterior Dirichlet problem:
\begin{itemize}
\item $\Delta u = 0$ in $\Omega'$,
\item $u \in H^2(\Omega'_\rho)$, for any $\rho > \rho_0$,
\item $u$ satisfies the following generalized Sommerfeld radiation condition:
$$\lim_{\rho \to +\infty} \int_{|y| = \rho}\left( G^q_\zeta(x,y)\frac{\partial u}{\partial \nu_y}(y)-u(y)\frac{\partial G^q_\zeta}{\partial \nu_y}(x,y) \right)d\sigma(y) = 0, \quad \text{a.e. } x \in \Omega',$$
\item $u|_{\partial \Omega'} = f.$
\end{itemize}
\end{lemma}

This is a slight generalization of \cite[Lemma A.6]{nachman1988}, and the proof can be obtained with the same approach.\smallskip

We also need to establish some jump formulas for the single and the double layer potentials.
\begin{lemma}\label{lem:jumps}
Let $\zeta\in\C^d$ be such that $|\zeta|\ge c_1$, where $c_1$ is given by Lemma~\ref{lem:Sqbound}, $f \in H^{-1/2}(\partial \Omega)$ and $u = S^q_\zeta f$.
Then the nontangential limits $\partial u/\partial \nu_+$ (resp.\ $\partial u/\partial \nu_-$) of $\partial u/\partial \nu$ as the boundary $\partial \Omega$ is approached from outside (resp.\ inside) $\Omega$ satisfy:
\begin{equation}\label{id:jumps2}
\frac{\partial u}{\partial \nu_-}-\frac{\partial u}{\partial \nu_+} = f, \quad \text{a.e. on } \partial \Omega.
\end{equation}  
\end{lemma}
\begin{proof}
The proof for $q= 0$, $f \in H^{1/2}(\partial \Omega)$ and $\partial \Omega \in C^{1,1}$ is given in \cite[Lemma 2.4]{nachman1988}. Based on \cite{verchota1984}, using the same arguments as in \cite[Lemma 7.1]{Nachman1996}, this can be extended to $f \in H^{-1/2}(\partial \Omega)$ and $\partial \Omega$ Lipschitz. For $q \neq 0$, note that if $\tilde \Omega$ is a bounded $C^{1,1}$ neighborhood of $\Omega$, the right hand side of  identity \eqref{id:ssq} is in $H^2(\tilde \Omega)$ by Lemma~\ref{lem:Sqbound}. Thus
\begin{equation*}
\frac{\partial}{\partial \nu_-}\left(S^q_\zeta - S_\zeta \right)f = \frac{\partial}{\partial \nu_+}\left(S^q_\zeta - S_\zeta\right)f,
\end{equation*}
and the proof follows from the  corresponding result for $q=0$.
\end{proof}

\begin{lemma}\label{lem:jump}
Let $\zeta\in\C^d$ be such that $|\zeta|\ge c_1$, where $c_1$ is given by Lemma~\ref{lem:Sqbound},  $f \in H^{1/2}(\partial \Omega)$ and $v = D^q_\zeta f$. Then, the nontangential limits $v_+ (v_-)$ of $v$ as we approach the boundary from outside (respectively inside) $\Omega$ exist and satisfy:
\begin{equation*}
v_\pm (x) = \pm \frac{1}{ 2} f(x) + B^q_\zeta f(x), \quad \text{for a.e. } x \in \partial \Omega.
\end{equation*}
\end{lemma}
\begin{proof}
For $q=0$ this was proved in \cite[Lemma 2.5]{nachman1988} for $f \in H^{3/2}(\partial \Omega)$ and $\partial \Omega \in C^{1,1}$ but using the results of \cite{verchota1984} it can be extended to $f \in H^{1/2}(\partial \Omega)$ and $\partial \Omega$ Lipschitz. For $q \neq 0$, using identity \eqref{id:ggq} we have
\begin{align*}
D^q_\zeta f(x) &= D_\zeta^0 f(x) - \int_{\R^d}G^q_\zeta(x,z)q(z)D_\zeta^0 f(z)dz,\quad x \in \R^d \setminus \partial \Omega, \\
B^q_\zeta f(x) &= B_\zeta^0 f(x) - \int_{\R^d}G^q_\zeta(x,z)q(z)D_\zeta^0 f(z)dz, \quad x \in \partial \Omega.
\end{align*}
The result for $q=0$ directly gives
\begin{align*}
v_{\pm}(x) &= \pm \frac 1 2 f(x)+B_\zeta^0 f(x)- \int_{\R^d}G^q_\zeta(x,z)q(z)D_\zeta^0 f(z)dz  = \pm \frac 1 2 f(x)+B^q_\zeta f(x),
\end{align*}
as desired.
\end{proof}

We now come to the main result of the section.
\begin{proposition}\label{prop:Sqinv}
Let $D>0$ and  $\Omega$ be a bounded domain with Lipschitz boundary such that $\Omega' = \R^d \setminus \Omega$ is connected. Let $\zeta\in\C^d$ be such that $|\zeta|\ge c_1$, where $c_1$ is given by Lemma~\ref{lem:Sqbound}, and $q \in L^{\infty}(\R^d)$ satisfy the assumptions at the beginning of the section. 

Then the operator $S^q_\zeta \colon H^{-1/2}(\partial \Omega) \to H^{1/2}(\partial \Omega)$ is invertible with bounded inverse and
\[
\|(S^q_\zeta)^{-1}\|\le c\bigl(\Omega,R,D,\|(-\Delta+q)^{-1}\|_{H^{-1}(\Omega)\to H^1_0(\Omega)}\bigr).
\]
\end{proposition}
\begin{proof}
The proof is inspired by \cite[Lemma A.7]{nachman1988}.

For the injectivity we follow the argument at the beginning of the proof of \cite[Theorem 1.6, \S6]{nachman1988}. Assume $S^q_\zeta f = 0$ on $\partial \Omega$. Then $u = S^q_\zeta f$ is a Dirichlet eigenfunction of $-\Delta + q$ in $\Omega$, and so $u = 0$ in $\Omega$ by the assumptions on $q$. On the other hand, $u$ solves the exterior Dirichlet problem of Lemma \ref{lem:ext} with homogeneous conditions (the Sommerfeld radiation condition can be checked as in \cite[Lemma 2.4]{nachman1988}) and so $u = 0$ in $\Omega'$. This means that both $\partial u /\partial \nu^+$ and $\partial u /\partial \nu^-$ vanish on $\partial \Omega$. By the jump formula \eqref{id:jumps2}, $f$ mush vanish as well.

 In order to prove surjectivity we construct  an inverse explicitly. 
Thanks to Lemma~\ref{lem:ext}, we can define the Dirichlet-to-Neumann map $\Lambda_q^+ f = \frac{\partial u}{\partial \nu_+}$ for $f\in H^{1/2}(\partial\Omega)$, where $u$ is the unique solution to the exterior problem given in Lemma~\ref{lem:ext}. Now applying Green's formula to $G^q_\zeta(x,y)$ and $u(y)$ in $\Omega'_\rho$ and letting $\rho \to +\infty$ we obtain, using the generalized radiation condition:
\begin{equation*}
u(x) = -\int_{\partial \Omega}\left( G^q_\zeta(x,y)\frac{\partial u}{\partial \nu_y}(y)-u(y)\frac{\partial G^q_\zeta}{\partial \nu_y}(x,y) \right)d\sigma(y) \quad \text{for a.e. } x \in \Omega'.
\end{equation*}
Taking the trace on the boundary and using Lemma \ref{lem:jump} we obtain
\begin{equation}\label{eq:out}
S^q_\zeta \Lambda_q^+ = -\frac{1}{2}I +B^q_\zeta,
\end{equation} 	
where $I$ denotes the identity operator on $H^{1/2}(\partial\Omega)$.

Now let $u'$ be the unique solution of $(-\Delta + q) u' = 0$ in $\Omega$ and $u'|_{\partial \Omega} =f$. Applying again Green's formula to $G^q_\zeta(x,y)$ and $u'(y)$ for $x \in \Omega$ we obtain
\begin{equation*}
u'(x) = \int_{\partial \Omega}\left( G^q_\zeta(x,y)\frac{\partial u'}{\partial \nu_y}(y)-u'(y)\frac{\partial G^q_\zeta}{\partial \nu_y}(x,y) \right)d\sigma(y).
\end{equation*}
Letting $x$ approach the boundary nontangentially inside $\Omega$ we find, using again Lemma \ref{lem:jump},
\begin{equation}\label{eq:in}
S^q_\zeta \Lambda_q = \frac{1}{2}I +B^q_\zeta.
\end{equation}
Now, identities \eqref{eq:out} and \eqref{eq:in} give $S^q_\zeta (\Lambda_q - \Lambda_q^+) = I$, which show that $S^q_\zeta$ is surjective and that its inverse is $\Lambda_q - \Lambda_q^+$.

Finally, the boundedness of $(S^q_\zeta)^{-1}$ comes from the estimates
\begin{align*}
\|\Lambda_q f\|_{H^{-1/2}(\partial \Omega)} &\leq c \|f\|_{H^{1/2}(\partial \Omega)},\\
\|\Lambda^+_q f\|_{H^{-1/2}(\partial \Omega)} &\leq c \|f\|_{H^{1/2}(\partial \Omega)},
\end{align*}
which follow from  classical elliptic estimates, where $c>0$ depends only on $\Omega$, $R$, $D$ and $\|(-\Delta+q)^{-1}\|_{H^{-1}(\Omega)\to H^1_0(\Omega)}$.
\end{proof}

\bibliographystyle{plain}
\bibliography{CS}

\end{document}